\documentclass[11pt,a4paper,twoside]{amsart}
\usepackage{amsmath,amssymb,amsfonts,amsthm}
\usepackage{multirow}
\usepackage{soul,xcolor}
\usepackage{cancel}
\usepackage{hyperref}
\usepackage{graphicx}
\usepackage{lineno}

\def\d{\partial}

\def\eps{\epsilon}

\newcommand{\cM}{{\mathcal M}}
\newcommand{\cG}{{\mathcal G}}
\newcommand{\cS}{{\mathcal S}}

\def\<{\langle}
\def\>{\rangle}

\def \cL{{\mathcal{L}}}

\def\E{\mathbb{E}}

\DeclareMathOperator{\tr}{tr}
\DeclareMathOperator{\vol}{vol}
\newtheorem{The}{Theorem}[section]

\newtheorem{Lem}[The]{Lemma}

\theoremstyle{definition}

\newtheorem{Rem}[The]{Remark}

\begin{document}
\title{Stochastic perturbation of geodesics on the manifold of Riemannian metrics}
\author{ Ana Bela Cruzeiro and Ali Suri }
\address{Ana Bela Cruzeiro, GFM, Dep. Mathematics Instituto Superior técnico, Av. Rovisco Pais, 1049-001 Lisboa, Portugal}
\email{ana.cruzeiro@tecnico.ulisboa.pt}
\address{Ali Suri, Universit\"{a}t Paderborn, Warburger Str. 100,
33098 Paderborn, Germany}
\email{asuri@math.upb.de}
\maketitle {\hspace{2.5cm}}

\begin{abstract} 
In this paper, using diffusion processes, we compute the evolution equation
on the manifold of Riemannian metrics for the Lagrangian induced by the $L^2$ stochastic kinetic energy functional.


\end{abstract}

\pagestyle{headings} \markright{Stochastic perturbation of geodesics on the manifold of Riemannian metrics}
%
%
\section{Introduction}
The study of the infinite-dimensional manifold of Riemannian metrics  traces back to foundational works in global analysis and geometric structures on function spaces. DeWitt \cite{DeWitt} used the manifold of Riemannian metrics to develop a Hamiltonian formulation of general relativity. A major breakthrough came in the 1970s with the seminal work of David Ebin \cite{Ebin70}, who rigorously established the manifold structure of the space of all Riemannian metrics.  Manifold of Riemannian metrics as  the configuration space for general relativity with the DeWitt metric was also considered in \cite{Fischer-Marsden, Marsden-Ebin-Fischer}. Moreover, the manifold of Riemannian metrics appears in Teichm\"{u}ller theory \cite{Tromba} and shape analysis \cite{BBHM}.

A key tool introduced in this framework is the Ebin  metric, a weak Riemannian metric on the space of Riemannian metrics, defined via an $L^2$ inner product of symmetric 2-tensors. This structure has been further explored by  Freed and  Groisser in \cite{Freed}, Gil-Medrano and Michor \cite{Med-Mic} and Clarke \cite{Clarke}, leading to deeper insights into its geometric and analytical properties.

On the other hand, Cipriano and Cruzeiro \cite{Cip-Cru} and  Arnaudon et al. \cite{A-C-C14}  developed a geometric and stochastic  approach to the Navier–Stokes equations using the infinite-dimensional structure of the group of volume preserving diffeomorphisms. Inspired by Arnold’s geometric interpretation of Euler’s equations for ideal fluids \cite{Arnold66}, they extended this geometric variational method to the framework of dissipative,  non-perfect (viscous) fluids by incorporating group valued semi-martingales. This approach was also applied to the compressible Navier-Stokes equations and to magnetohydrodynamics for charged viscous compressible fluids in \cite{CCR} and incompressible fluids at the presence of the Coriolis force on the torus in \cite{Suri-SEPC}.

Recall that the Euler equations for perfect fluids and the geodesic equations on the manifold of Riemannian metrics are naturally obtained  via variational principles. However,  Navier–Stokes equations and other dissipative systems need the use of stochastic processes in order to be   formulated via (stochastic) variational principles  (e.g., \cite{A-C-C14}). \color{black}  Motivated by the above works, we complete the picture by introducing a pointwise noise into the framework of the manifold of Riemannian metrics and computing the corresponding evolution equations via a similar stochastic Lagrangian.

%
\section{Geodesic equations on the manifold of Riemannian metrics}
In this section, we recall the fundamental background on the manifold of Riemannian metrics $\mathcal{M}$, as well as a special class of semimartingales whose drift takes values in $\mathcal{M}$..

Let $M$ be a compact $3$-dimensional oriented Riemannian manifold. We may consider the noncompact case by suitable assumptions at infinity.
Following the notation of \cite{Marsden-Ebin-Fischer}, and \cite{marsden}, let $\mathcal{S}_2(M)$ be the linear Fr\'{e}chet space of all (smooth) symmetric covariant 2-tensor fields on $M$. The Sobolev case can be treated similarly \cite{BBHM}. 
Then, the space of all Riemannian metrics on $M$, denoted by  $\mathcal{M}$, forms an open cone in $\mathcal{S}_2(M)$ in both the Fr\'{e}chet and Sobolev topologies.  This implies that the tangent bundle is given by $T\mathcal{M}=\mathcal{M}\times \mathcal{S}_2(M) $.
\subsection{The Ebin  metric}
For $g\in \mathcal{M}$ and $h,k\in T_g\mathcal{M}\simeq \mathcal{S}_2(M)$, the Ebin (or $L^2$) metric is defined by
\begin{equation}\label{L2 metric}
\mathcal{G}_g(h,k)=\int_M \tr_g(h\times k)\vol(g)
\end{equation}
where $\vol(g)=\sqrt{detg_{ij}}dx_1\wedge dx_2\wedge dx_3$, $h\times k=h g^{-1}k$, $\tr_g(h)=\tr(g^{-1}h)$ and consequently $\tr_g(h\times k):=\tr(g^{-1}hg^{-1}k)$. 
Furthermore, for $h_1,\dots,h_n\in T_g\mathcal{M}$, we set  $\tr_g(h_1\times \dots\times h_n)=\tr(g^{-1}h_1\dots g^{-1}h_n)$.

The geodesic equation with respect to the Ebin metric \eqref{L2 metric} is given by
\begin{equation}\label{geo eq L2}
\d^2_t g=\d_tg \times \d_tg +\frac{1}{4}\tr_g(\d_tg  \times \d_tg)g - \frac{1}{2}\tr_g(\d_tg) \d_t g.
\end{equation}
(See e.g. \cite{Med-Mic}, Lemma 2.3.)
\subsection{Semi-martingales with $\mathcal{M}$-valued drift}
In the sequel, we consider the diffusion process on $\cM$ given by
\begin{equation}\label{diff proc} 
dg(t,x) = \sqrt{\nu} \sum_{i=1}^n H_{i}(x) dW_{i}(t) + K_\omega(t,x) dt
\end{equation} 
where $H_{i}$, $1\leq i\leq n$ are  symmetric 2-tensors on $M$. Additionally, $\nu$ is a nonnegative constant, $K_\omega \in C^\infty([a,b], \cS_2(M))$, and $W_{i}(t) $ for $1 \leq i \leq n$ are independent 1-dimensional Brownian motions. We will drop the random variable $\omega$ and write $K$ for $K_\omega$. Furthermore, we require that the drift term in \eqref{diff proc} (i.e., its expectation) takes values in $\mathcal{M}$.
\begin{Rem}
In hydrodynamics, there are canonical candidates for the counterpart of the process \eqref{diff proc}, which can lead to known correction terms in the evolution equations representing viscosity or friction \cite{A-C-C14, Cip-Cru, Suri-SEPC}.
However, for the setting of the manifold of Riemannian metrics, such canonical candidates are not yet clearly identified.\\
Some natural choices include $H_i = L_{X_i}g$, where $\{X_i\}_{i\in\mathbb{N}}$ is a family of vector fields on $M$.
Another approach is to use the decomposition of the tangent space at $g \in \mathcal{M}$ into 
\[
T_g  \mathcal{M}=V_0(g)\oplus V_1(g)
\]
where $V_0(g)=\{h\in T_g  \mathcal{M}~;~\tr_g(h)=0\}$ and $V_1(g)=C^{\infty}(M)g$. Recall that for $h\in T_g  \mathcal{M}$, the traceless part is given by  $h_0:=h-\frac{1}{3}\tr_g(h)g$.

Alternatively, one may consider symmetric 2-tensors $\{H_{ij}\}_{1 \leq i,j \leq 3}$ on $M$ whose $(i,j)$- and $(j,i)$-components are equal to $1$, with all other components vanishing. This formulation also allows for using different amplitudes $\nu_{ij}$ instead of a single constant $\nu$ as in \eqref{diff proc}. For example, when $\nu_{ij} = 0$ for $i \neq j$, $\nu_{ii} = 1$, and $W_{ii} = W$ for $1 \leq i, j \leq 3$, the perturbation occurs in a single direction, namely along $\mathrm{diag}(1,1,1)$.

In principle, even an infinite number of terms (i.e., $n = \infty$) can be employed, though this requires careful attention to the convergence of the solution, as discussed in \cite{Cip-Cru}.
\end{Rem}
%
%
%
%
\section{Stochastic perturbation of geodesics on $(\cM,\cG)$}
In this section, using the semi-martingale \eqref{diff proc}, we study the stochastic perturbation of geodesics in $(\cM,\cG)$. 
\begin{Lem}
The curve $g^{-1}$  almost surely exists  and satisfies
\begin{equation}\label{g inverse}
dg^{-1}= -\sqrt{\nu}\sum_{i}g^{-1}H_{i}g^{-1}dW_{i}(t)  -g^{-1}\big( K  -  \nu \sum_{i}H_{i}g^{-1}H_{i}\big)g^{-1}dt
\end{equation}
\end{Lem}
\begin{proof}
For the smooth  vector fields  $H_{i},K$ on the Banach manifold $\cM$, the process \eqref{diff proc} has a solution and its drift belongs to $\cM$. 
Moreover using It\^{o} formula we have
\begin{eqnarray*}
0=d(g^{-1}g)=(dg^{-1})g + g^{-1}dg +\langle dg^{-1}, dg\rangle
\end{eqnarray*}
which implies that
\begin{eqnarray*}
dg^{-1}&=&-g^{-1}dg g^{-1} + \langle g^{-1}dgg^{-1},dg\rangle g^{-1}\\
&=&-\sqrt{\nu}\sum_{i}g^{-1}H_{i}g^{-1}dW_{i}(t) -g^{-1}Kg^{-1}dt +  \nu \sum_{i}g^{-1}H_{i}g^{-1}H_{i}g^{-1}dt\\
&=&-\sqrt{\nu}\sum_{i}g^{-1}H_{i}g^{-1}dW_{i}(t) -g^{-1}\big( K -  \nu \sum_{i}H_{i}g^{-1}H_{i}\big)g^{-1}dt.
\end{eqnarray*}
\end{proof}
The following lemma provides a formula for computing $d\vol(g)$, where $g$ is a solution of \eqref{diff proc} and $\vol(g)=\sqrt{\det(g_{ij})}dx_1\wedge dx_2\wedge dx_3$. Define $F(g)=\sqrt{\det(g_{ij})}$.
\small\begin{Lem}\label{Lem dvol}
The following holds true
\begin{equation}\label{dvol g }
d\vol(g)= \frac{1}{2}\Big(  \sqrt{\nu}\sum_{i}\tr_g(H_{i})dW_{i}(t) + \big[ \tr_g(K) + \frac{{\nu}}{2}\sum_{i}   H_{i}\tr_g(H_{i}) +\frac{1}{2}\tr_g(H_{i})^2\big]dt\Big)\vol(g).
\end{equation}
\end{Lem}
\begin{proof}
Using It\^{o} formula for the function $F(g)$ and the fact that 
\[
H_{i}F(g)=\frac{1}{2}\tr_g(H_{i})F(g)
\]
we have
\[
H_{i}H_{i}F(g)=\frac{1}{2}H_{i}(\tr_g(H_{i})F(g))=\frac{1}{2}\big(  H_{i}\tr_g(H_{i}) +\frac{1}{2}\tr_g(H_{i})^2\big)F(g).
\]
As a result we get
\begin{eqnarray*}
&&d\vol(g)=\Big(  \sqrt{\nu}(H_{i}F(g)dW_{i}(t) +KF(g)dt +  \frac{\nu}{2}\sum_{i}  H_{i}H_{i}F(g) dt\Big) dx_1\wedge dx_2\wedge dx_3\\
&=&\Big(  \frac{\sqrt{\nu}}{2}\sum_{i}\tr_g(H_{i})dW_{i}(t) +\frac{1}{2}\tr_g(K)dt +  \frac{{\nu}}{4}\sum_{i} \big(  H_{i}\tr_g(H_{i}) +\frac{1}{2}\tr_g(H_{i})^2\big)dt\Big)\vol(g)\\
&=& \frac{1}{2}\Big(  \sqrt{\nu}\sum_{i}\tr_g(H_{i})dW_{i}(t) + \big[ \tr_g(K) + \frac{{\nu}}{2}\sum_{i}   H_{i}\tr_g(H_{i}) +\frac{1}{2}\tr_g(H_{i})^2\big]dt\Big)\vol(g).
\end{eqnarray*}
\end{proof}
For any   $V(\cdot)\in C^\infty([a,b], \cS_2(T^*M))$ satisfying the conditions  $V(a)=V(b)=0$ consider the variations of \eqref{diff proc}   given by
\begin{equation}\label{diff proc variation}
g(t,s,x) = \sqrt{\nu}\int_a^t \sum_{i} H_{i}(x) dW_{i}(\tau)  + \int_a^t K(\tau) d\tau + sV(t)\quad ;\quad -\eps<s<\eps
\end{equation}
or equivalently
\begin{equation}\label{diff proc variation2} 
dg(t) = \sqrt{\nu} \sum_{i} H_{i}(x) dW_{i}(t) + K(t) dt + s \d_tV(t)dt.
\end{equation}
The following version of integration by parts will be used when computing the perturbed geodesic equations through a variational principle.
\begin{Lem}\label{Lem integration by parts}
\begin{eqnarray*}
&&\E\int_a^b \cG_g\big(K~,~\d_t V\big) dt = -  \E\int_a^b \cG_g \big( -2 K \times K + \d_tK +\frac{1}{2}\tr_g(K) K~,~V\big)dt \\
&&-{\nu}\sum_{i}\int_a^b \cG_g\big(  H_{i}\times H_{i}\times K + K \times H_{i}\times H_{i}   + H_{i}\times K\times H_{i}~,~ V)\\
&&  +   \frac{1}{2} \cG_g\Big(\frac{1}{2}\big(  H_{i}\tr_g(H_{i}) +\frac{1}{2}\tr_g(H_{i})^2\big)K~~{-}~~ \tr_g(H_{i}) \big(H_{i}\times K + K\times H_{i}\big)~,~V\Big) dt.
\end{eqnarray*}
\end{Lem}
\begin{proof}
Set $d(g^{-1}K):=dX^t_1$, $d(g^{-1}V):=dX^t_2$,
\[
dtr(g^{-1}Kg^{-1}V)=trd(g^{-1}Kg^{-1}V):=dX^t 
\]
and $d\vol(g):=dY^t$.
Then 
\begin{eqnarray*}
d\cG_g(K,V)dt = d\int_M \tr_g(KV)\vol(g)= \int_M (dX) Y +XdY + \langle dX,dY\rangle
\end{eqnarray*}
On the other hand,
\begin{eqnarray*}
&&dX=d\tr(X_1X_2)=\tr\big[ (dX_1) X_2+X_1dX_2 + \langle dX_1,dX_2\rangle\big]\\
&=&\tr\Big[\Big(-\sqrt{\nu}\sum_{i} g^{-1}H_{i}g^{-1}dW_{i}(t) -g^{-1}\big(K - \nu\sum_{i} H_{i}g^{-1}H_{i} \big)g^{-1}dt\Big) Kg^{-1}V\\
&& + g^{-1}(\d_tK)g^{-1}V\\
&& + g^{-1}K \Big(   -\sqrt{\nu}\sum_{i} g^{-1}H_{i}g^{-1}dW_{i}(t) -g^{-1} \big(K - \nu\sum_{i} H_{i}g^{-1}H_{i}\big)g^{-1}dt\Big)V\\
&& + g^{-1}Kg^{-1}\d_{t}V\\
&& +\nu\sum_{i} (g^{-1}H_{i}g^{-1}K)( g^{-1}H_{i}g^{-1}V)dt\Big]
\end{eqnarray*}
\begin{eqnarray*}
&=&-\sqrt{\nu}\sum_{i} \tr_g\big( H_{i}\times K\times V\big) dW_{i}(t) -\tr_g\big(K\times  K\times V\big)dt \\
&&+ \nu\sum_{i} \tr_g\big(H_{i}\times H_{i}\times K\times V)\big) dt + \tr_g\big((\d_tK)\times V\big)dt\\
&& -\sqrt{\nu}\sum_{i} \tr_g\big(K\times  H_{i}\times V\big) dW_{i}(t) - \tr_g\big(K\times K\times V\big)dt \\
&&+ \nu\sum_{i} \tr_g\big(K\times H_{i}\times H_{i}\times  V\big)dt +\tr_g\big(  K\times \d_{t}V\big) dt \\
&&+\nu\sum_{i} \tr_g\big( H_{i}\times K\times H_{i}\times V \big)dt \\
%
%
&=&-\sqrt{\nu}\sum_{i} \tr_g\big( H_{i}\times K\times V  + K\times  H_{i}\times V\big) dW_{i}(t)\\
&&+ \Big[ -2\tr_g\big(K \times K\times V\big)   + \tr_g\big((\d_tK)\times V\big)   +\tr_g\big(  K\times \d_{t}V\big)\\
&& + \nu\sum_{i} \tr_g\big(H_{i}\times H_{i}\times K\times V + K\times H_{i}\times H_{i} \times V  + H_{i}\times K\times H_{i}\times V \big)\Big]dt 
\end{eqnarray*}
Now, using Lemma \ref{Lem dvol}, we obtain
\begin{eqnarray*}
XdY &=& \frac{1}{2}\tr_g(K\times V)\Big( \sqrt{\nu}\sum_{i}\tr_g(H_{i})dW_{i}(t)  +\big[\tr_g(K) \\
&&    +  \frac{{\nu}}{2}\sum_{i}   H_{i}\tr_g(H_{i}) +\frac{1}{2}\tr_g(H_{i})^2\big]dt\Big)\vol(g)
\end{eqnarray*}
and finally
\begin{eqnarray*}
\langle dX,dY\rangle=-\frac{\nu}{2} \sum_{i}\tr_g\big( H_{i}\times K\times V  + K\times  H_{i}\times V\big) \tr_g(H_{i})dt.
\end{eqnarray*}
Combining all these results and using the fact that the expectation of the It\^{o} integral $\E\int_a^b Z dW^t_{ij} = 0$ for any square-integrable semi-martingale $Z$, we obtain
\begin{eqnarray*}
&&   \E\int_a^b d\cG_g(K,V)dt\\
&=& \E\int_a^b\int_M \Big[ -2\tr_g\big(K \times K\times V\big)   + \tr_g\big((\d_tK)\times V\big)   +\tr_g\big(  K\times \d_{t}V\big)\\
&& +\nu\sum_{i} \tr_g\big(H_{i}\times H_{i}\times K\times V + K\times H_{i}\times H_{i} \times V  + H_{i}\times K\times H_{i}\times V \big)\Big] \vol(g)dt\\
&& +\frac{1}{2} \E\int_a^b\int_M\tr_g(K\times V)\Big(    \tr_g(K)    +  \frac{\sqrt{v}}{2}\sum_{i}   H_{i}\tr_g(H_{i}) +\frac{1}{2}\tr_g(H_{i})^2\Big)\vol(g)dt\\
&& - \frac{\nu}{2} \sum_{i}\E\int_a^b\int_M \tr_g\big( H_{i}\times K\times V  + K\times  H_{i}\times V\big) \tr_g(H_{i})\vol(g)dt.
\end{eqnarray*}

%
As a result we have
\begin{eqnarray*}
&&\E\int_a^b \cG_g\big(K~,~\d_tV\big) dt = -  \E\int_a^b \cG_g \big( -2 K \times K + \d_tK +\frac{1}{2}\tr_g(K) K~,~V\big)dt \\
&&-{\nu}\sum_{i}\E\int_a^b \cG_g\big(  H_{i}\times H_{i}\times K + K \times H_{i}\times H_{i}   + H_{i}\times K\times H_{i}~,~ V)\\
&&  +   \frac{1}{2} \cG_g\Big(\frac{1}{2}\big(  H_{i}\tr_g(H_{i}) +\frac{1}{2}\tr_g(H_{i})^2\big)K~~{-}~~ \tr_g(H_{i}) \big(H_{i}\times K + K\times H_{i}\big)~,~V\Big) dt.
\end{eqnarray*}
\end{proof}
\subsection{ Stochastic kinetic energy functional}
Let $\mathcal{S}(\cM)$ be the space of all  semi-martingales with $\mathcal{M}$-valued drift. Consider the following stochastic kinetic energy functional on $\mathcal{S}(\cM)$ 
\begin{equation*}\label{stochastic energy functional}
J^{\mathcal{\cG}}(\xi) =  \frac{1}{2}  \E \int_a^b  \cG_g \big( D_t\xi   , D_t\xi \big)      dt
\end{equation*}
where $D_t$ represents the generalized mean derivative. Recall that for a semi-martingale $\xi$ with values in $\cM$, if for any $f \in C^\infty(\cM)$ and $a \leq t \leq b$,  the process
\[
N_t^f:=f(\xi(t))-f(\xi(0)-\frac{1}{2}\int_a^t \mathrm{Hess}f(\xi(\tau))\big(d\xi(\tau),d\xi(\tau)\big)d\tau -\int_a^t A(\tau)f(\xi(\tau))d\tau
\]
is a real valued local semi-martingale, where $A$ is a symmetric 2-tensor, we define $D_t\xi:=A$. 
Roughly speaking, the generalized derivative of the process \eqref{diff proc} gives the bounded variation part $K$.

\begin{The}{}
The semi-martingale $g(\cdot)$ defined by \eqref{diff proc}  is a critical point of $J^\cG$ if and only if the deterministic  path $K \in C^{1}([a, b];\cS_2(T^*M) )$ satisfies the following equations
\begin{eqnarray}\label{modified Euler Lagrange}
\left\{ \begin{array}{ll}  
D_tg =K(t)\\
\d_tK=  K \times K + \frac{1}{4}  \tr_g(K\times K)g   -\frac{1}{2}\tr_g(K)K  -\nu\cL(K)
\end{array}\right.
\end{eqnarray}
where 
\begin{eqnarray*}
\cL(K)&:=& \sum_{i} \Big(  H_{i}\times H_{i}\times K + H_{i}\times K\times H_{i} + K \times H_{i}\times H_{i}   \\
&&  +    \big(\frac{1}{4}\big(  H_{i}\tr_g(H_{i}) +\frac{1}{2}\tr_g(H_{i})^2\big)K ~ - ~ \frac{1}{2}\tr_g(H_{i}) \big(H_{i}\times K + K\times H_{i}\big)~\Big).    
\end{eqnarray*}
\end{The}
\begin{proof}
For the variation \eqref{diff proc variation} we have
\begin{eqnarray*}
\d_s|_{s=0}J^{\cG}(g(\cdot,s)) &=&  \frac{1}{2}  \E \d_s|_{s=0}\int_a^b  \cG_g \big( D_tg(t,s)   , D_tg(t,s) \big)      dt\\
&=&  \frac{1}{2}  \E \d_s|_{s=0}\int_a^b  \int_M\tr_g \big( D_tg(t,s)   , D_tg(t,s) \big)   \vol(g)   dt\\
&=&  \frac{1}{2}  \E \int_a^b \d_s|_{s=0} \int_M\tr \big( g^{-1}(t,s)D_tg(t,s)  g^{-1}(t,s)D_tg(t,s) )    \big)   \vol(g)   dt
\end{eqnarray*}

\begin{eqnarray*}
&=&   \frac{1}{2}\E \int_a^b \int_M \Big[  -2\tr \big( g^{-1}V g^{-1}K g^{-1}K    \big) + 2\tr \big( g^{-1}(\d_tV)g^{-1}K  \big)\\
&&+ \tr \big( g^{-1}Kg^{-1}K  \big)\frac{1}{2}\tr(g^{-1}V )\Big]\vol(g)   dt\\
&=&   \E \int_a^b \int_M \Big[  -\tr \big( g^{-1}V g^{-1}K g^{-1}K    \big) + \tr \big( g^{-1}(\d_tV )g^{-1}K  \big)\\
&& + \frac{1}{4}\tr \big( g^{-1}Kg^{-1}K  \big)\tr(g^{-1}{gg^{-1}}V )\Big]\vol(g)   dt\\
&=&   \E \int_a^b  -\cG_g \big( K \times K~,~  V  \big) + \cG_g \big(  K ~,~  \d_t V   \big)   + \frac{1}{4} \cG_g \big( \tr_g(K\times K)g ~,~V  \big)   dt\\
&\stackrel{\textrm{Lemma} ~\ref{Lem integration by parts}}{=}&   \E \int_a^b  -\cG_g \big( K \times K~,~  V   \big) + \frac{1}{4} \cG_g \big( \tr_g(K\times K)g ~,~V  \big)   dt\\
&&-  \E\int_a^b \cG_g \big( -2 K \times K + \d_tK +\frac{1}{2}\tr_g(K) K~,~V \big)dt \\
&&-{\nu}\sum_{i}\E\int_a^b \cG_g\big(  H_{i}\times H_{i}\times K + K \times H_{i}\times H_{i}   + H_{i}\times K\times H_{i}~,~ V )\\
&&  +   \frac{1}{2} \cG_g\Big(\frac{1}{2}\big(  H_{i}\tr_g(H_{i}) +\frac{1}{2}\tr_g(H_{i})^2\big)K~~{-}~~ \tr_g(H_{i}) \big(H_{i}\times K + K\times H_{i}\big)~,~V \Big) dt.
\end{eqnarray*}
\begin{eqnarray*}
&=&   \E\int_a^b  \cG_g \big( k \times K ~ +~ \frac{1}{4}  \tr_g(K\times K)g   -\d_tK -\frac{1}{2}\tr_g(K) K~,~V \big)dt \\
&&-\nu\sum_{i}\E\int_a^b \cG_g\big(  H_{i}\times H_{i}\times K + K \times H_{i}\times H_{i} + H_{i}\times K\times H_{i}~,~ V )\\
&&  +   \frac{1}{2} \cG_g\Big(\frac{1}{2}\big(  H_{i}\tr_g(H_{i}) +\frac{1}{2}\tr_g(H_{i})^2\big)K~-~ \tr_g(H_{i}) \big(H_{i}\times K + K\times H_{i}\big)~,~V \Big) dt.
\end{eqnarray*}
Since the metric $\cG$ is non-degenerate and the variation is arbitrary then we conclude that
\begin{eqnarray*}
\d_tK &=& K \times K ~ +~ \frac{1}{4}  \tr_g(K\times K)g   -\frac{1}{2}\tr_g(K) K \\
&&-\nu\sum_{i} \Big(  H_{i}\times H_{i}\times K +H_{i}\times K\times H_{i} + K \times H_{i}\times H_{i}  \\
&&  +    \big(\frac{1}{4}\big(  H_{i}\tr_g(H_{i}) +\frac{1}{2}\tr_g(H_{i})^2\big)K~-~ \frac{1}{2}\tr_g(H_{i}) \big(H_{i}\times K + K\times H_{i}\big)~\Big).
\end{eqnarray*}
\end{proof}
\begin{Rem}
\textbf{i.} Although the correction term $\cL(K)$ shares similarities with the Laplace-Beltrami operator
\[
\Delta_B K=\sum_{i} \nabla_{H_{i}}\nabla_{H_{i}} K -\nabla_{\nabla_{H_{i}}H_{i}}K
\]
or 
\[
\Delta K=\Delta_B K- \sum_{i} R_g(K,H_{i})H_{i}
\] the terms do not match exactly. There is ongoing research aimed at modifying the process \eqref{diff proc} or the $L^2$ metric  and to provide a geometric meaning of the second-order operator $\cL(K)$.\\
\textbf{ii.} The above discussion is true for $n$-dimensional manifolds $M$ and arbitrary vector fields $\{H_{i}\}$ on $\cM$.\\
\textbf{iii.} When $\nu=0$ the equations \eqref{modified Euler Lagrange} reduce to the deterministic geodesic equation \eqref{geo eq L2}.
\end{Rem}
\noindent\textbf{Acknowledgements}.
The first author acknowledges the support of Portuguese FCT grant UIDB/00208/2023.
The second author was supported by  the Deutsche Forschungsgemeinschaft (DFG, German Research Foundation), project  517512794.

%
%

\end{document}